\documentclass[11pt]{article}
\usepackage{iftex}
\ifPDFTeX%
  \usepackage[T1]{fontenc}
  \usepackage{lmodern}
\else
  \usepackage{fontspec}
\fi
\usepackage[margin=1in]{geometry}
\usepackage{amsmath,amssymb,amsthm,array,booktabs}
\usepackage[explicit]{titlesec}

\newtheorem{thm}{Theorem}
\newtheorem{lem}[thm]{Lemma}
\newtheorem{cor}[thm]{Corollary}

\titleformat{\subsection}[runin]
  {\normalfont\large\bfseries}
  {\thesubsection}{1em}{#1.}
\titlespacing*{\subsection}{0pt}
  {3.25ex plus 1ex minus .2ex}{.6em}

\title{Equivalence of Restricted and Full Division}
\author{Naoya~Kato}
\date{}

\begin{document}
\maketitle

\begin{abstract}
We prove over $\mathrm{PA}^{-}$ that the restricted Euclidean division
principle $\mathrm{pa17}^{-}$, which asserts division only for co-prime
pairs, is equivalent to full division $\mathrm{pa17}$. A quotient-preserving
transformation reduces any division problem to one for a co-prime pair,
answering Visser's Open Question~8.29. Appendix~A derives the predecessor
principle $\mathrm{pa14}$ and the subtraction principle $\mathrm{pa15}$ from
full division over $\mathrm{PA}^{-}_{\mathrm{jer}}$. Appendix~B derives
$\mathrm{pa15}$ from restricted division and proves the equivalence over
that base theory.
\end{abstract}

\section{Introduction}

Visser~\cite[Open Question~8.29]{visser} asked whether, over
$\mathrm{PA}^{-}$, the restricted division principle $\mathrm{pa17}^{-}$
for co-prime pairs implies the full Euclidean Division Principle
$\mathrm{pa17}$. We prove
\[
\mathrm{PA}^{-}\vdash
\bigl(\mathrm{pa17}^{-}\leftrightarrow\mathrm{pa17}\bigr).
.\]
For background on weak arithmetic and its model theory, see
Hájek\textendash{}Pudlák~\cite{hajek-pudlak} and Kaye~\cite{kaye}.

The two principles therefore have the same reverse-mathematical strength
over $\mathrm{PA}^{-}$, so results formulated using either transfer to the
other.

Appendix~A derives the predecessor principle $\mathrm{pa14}$ and subtraction
principle $\mathrm{pa15}$ from $\mathrm{pa17}$ over
$\mathrm{PA}^{-}_{\mathrm{jer}}$. Appendix~B proves the same equivalence over
this weaker base theory:
\[
\mathrm{PA}^{-}_{\mathrm{jer}}\vdash
\bigl(\mathrm{pa17}^{-}\leftrightarrow\mathrm{pa17}\bigr).
\]
There, restricted division first yields pa15, allowing us to apply the main
equivalence; the converse follows from the definition of co-primality.

The key step is the invariance
\[
\operatorname{euc}(x,y)
\leftrightarrow
\operatorname{euc}(tx,ty+s),
\]
for $x\ne0$ and $0\le s<t$: it preserves the quotient and sends a remainder
$r$ to $tr+s$. Taking $t=x$, $s=1$, and $z+1=xy$ gives
\[
x^2y^2=(xy+1)z+1,
\]
so $(x^2,xy+1)$ is co-prime. Division for this pair then yields a quotient
and remainder for dividing $y$ by $x$.

\section{The Base Theory and the Division Principles}

\subsection{Language and axioms}

Let the base language be
\[
\mathcal L_0=\{0,1,+,\cdot,\le\}.
\]
We use $x<y$ as an abbreviation for $x\le y\land x\ne y$, and put
$2:=1+1$ and $x^2:=x\cdot x$. Following the numbering of the axioms and the
definition of $\mathrm{PA}^{-}$ in Visser~\cite[Section~3.1]{visser},
$\mathrm{PA}^{-}$ consists of pa1--pa13 and pa15 below. All free variables in
the axioms are understood to be universally quantified. We write
$\mathrm{PA}^{-}_{\mathrm{jer}}$ for the subsystem axiomatized by pa1--pa13;
thus $\mathrm{PA}^{-}$ is $\mathrm{PA}^{-}_{\mathrm{jer}}$ together with
pa15.
\begin{align*}
\mathrm{pa1}:\quad &x+0=x,\\
\mathrm{pa2}:\quad &x+y=y+x,\\
\mathrm{pa3}:\quad &(x+y)+z=x+(y+z),\\
\mathrm{pa4}:\quad &x\cdot1=x,\\
\mathrm{pa5}:\quad &x\cdot y=y\cdot x,\\
\mathrm{pa6}:\quad &(x\cdot y)\cdot z=x\cdot(y\cdot z),\\
\mathrm{pa7}:\quad &x\cdot(y+z)=x\cdot y+x\cdot z,\\
\mathrm{pa8}:\quad &x\le y\lor y\le x,\\
\mathrm{pa9}:\quad &(x\le y\land y\le z)\rightarrow x\le z,\\
\mathrm{pa10}:\quad &x+1\not\le x,\\
\mathrm{pa11}:\quad &y\le x\rightarrow(y=x\lor y+1\le x),\\
\mathrm{pa12}:\quad &y\le x\rightarrow y+z\le x+z,\\
\mathrm{pa13}:\quad &y\le x\rightarrow y\cdot z\le x\cdot z,\\
\mathrm{pa15}:\quad &y\le x\rightarrow\exists z\,(x=y+z).
\end{align*}

The predecessor principle pa14, which is not assumed in $\mathrm{PA}^{-}$,
is
\[
\mathrm{pa14}:\quad
\forall x\bigl(x=0\lor\exists y\,(x=y+1)\bigr).
\]

\subsection{Basic properties used in the proof}

As standard elementary consequences of pa1--pa13 and pa15, we shall use the
following properties.

\begin{lem}\label{lem:basic}
In $\mathrm{PA}^{-}$, $\le$ is a linear order, and the following hold.
\begin{enumerate}
\item If $a+c\le b+c$, then $a\le b$. In particular, if
      $a+c=b+c$, then $a=b$.
\item $a\cdot0=0$, $0\le a$, and $a\ne0\leftrightarrow0<a$.
\item $a<b\leftrightarrow a+1\le b$. In particular, if $a>0$, then
      $1\le a$.
\item If $a\le b$, then there exists $d$ such that $b=a+d$.
      If $a<b$, then this $d$ is positive.
\end{enumerate}
\end{lem}

Moreover, by pa12 and the order cancellation property above, addition also
preserves strict inequalities. Thus, if $d>0$, then $a+d>a$.

\subsection{The two division principles}

Following Visser~\cite[Section~8.8]{visser}, define
\[
\operatorname{euc}(a,b)
\;:\Longleftrightarrow\;
\exists q\,\exists r\,(b=aq+r\land r<a),
\]
\[
\operatorname{cop}(a,b)
\;:\Longleftrightarrow\;
\exists c\,\exists d\,(ad=bc+1).
\]
The formula $\operatorname{euc}(a,b)$ expresses that $b$ can be divided by
$a$ with remainder. A pair satisfying $\operatorname{cop}(a,b)$ is called a
co-prime pair. Here $c$ and $d$ are also nonnegative. By
Lemma~\ref{lem:basic}, the remainder condition $r<a$ is equivalent to
$0\le r<a$.

The restricted Euclidean Division Principle is
\[
\mathrm{pa17}^{-}:\quad
\forall a\,\forall b\,
\bigl(\operatorname{cop}(a,b)\rightarrow\operatorname{euc}(a,b)\bigr),
\]
and the full Euclidean Division Principle is
\[
\mathrm{pa17}:\quad
\forall a\,\forall b\,
\bigl(a\ne0\rightarrow\operatorname{euc}(a,b)\bigr).
\]

\section{Quotient Preservation and Transfer to Co-prime Pairs}

\begin{thm}[A quotient-preserving division transformation]\label{thm:affine-invariance}
\[
\mathrm{PA}^{-}\vdash
\forall x\,\forall y\,\forall s\,\forall t\,
\Bigl((x\ne0\land s<t)\rightarrow
\bigl(
\operatorname{euc}(x,y)
\leftrightarrow
\operatorname{euc}(tx,ty+s)
\bigr)\Bigr).
\]
\end{thm}

\begin{proof}
Assume $x\ne0$ and $s<t$.

First suppose $\operatorname{euc}(x,y)$, and let
\[
y=xq+r,\qquad r<x.
\]
Since $s<t$ and, by pa11, $r+1\le x$, we have
\[
tr+s<tr+t=t(r+1)\le tx.
\]
Thus
\[
ty+s=(tx)q+(tr+s),\qquad tr+s<tx,
\]
and $\operatorname{euc}(tx,ty+s)$ holds. The quotient here is the same $q$
as in the original division.

Conversely, suppose $\operatorname{euc}(tx,ty+s)$, and let
\begin{equation}
ty+s=txq+u,\qquad u<tx.
\label{eq:affine-division}
\end{equation}
We first show that $xq\le y$. If $y<xq$, then pa11 gives
$y+1\le xq$. Multiplying this inequality by $t$ and using
equation~\eqref{eq:affine-division}, we obtain
\[
ty+t\le txq\le ty+s.
\]
It follows that $t\le s$, contradicting $s<t$. Hence $xq\le y$.

By pa15, there exists $r$ such that
\begin{equation}
y=xq+r.
\label{eq:affine-transfer}
\end{equation}
Suppose that $x\le r$. Equation~\eqref{eq:affine-transfer} gives
$ty+s=txq+tr+s$. Comparing this with
equation~\eqref{eq:affine-division}, we obtain
\[
tr+s=u.
\]
By pa13, $tx\le tr\le tr+s=u$, contradicting $u<tx$. Hence $r<x$, and
$\operatorname{euc}(x,y)$ holds.
\end{proof}

\begin{cor}[Transfer to a co-prime pair]\label{cor:square-transfer}
\[
\mathrm{PA}^{-}\vdash
\forall x\,\forall y\,
\Bigl((1<x\land0<y)\rightarrow
\Bigl(
\operatorname{cop}(x^2,xy+1)
\land
\bigl(
\operatorname{euc}(x,y)
\leftrightarrow
\operatorname{euc}(x^2,xy+1)
\bigr)
\Bigr)\Bigr).
\]
\end{cor}

\begin{proof}
The equivalence of the division statements follows from
Theorem~\ref{thm:affine-invariance} by taking $t=x$ and $s=1$.

Now assume $1<x$ and $0<y$. Then $1\le xy$, so by pa15 there exists $z$
such that
\[
z+1=xy.
\]
Therefore
\[
x^2y^2={(z+1)}^2=(z+2)z+1=(xy+1)z+1.
\]
Thus $c=z$ and $d=y^2$ witness
$\operatorname{cop}(x^2,xy+1)$.
\end{proof}

\begin{cor}\label{cor:equivalence}
\[
\mathrm{PA}^{-}\vdash
\bigl(\mathrm{pa17}^{-}\leftrightarrow\mathrm{pa17}\bigr).
\]
\end{cor}

\begin{proof}
First assume $\mathrm{pa17}$. If $ad=bc+1$, then $a\ne0$, so
$\operatorname{cop}(a,b)$ implies $\operatorname{euc}(a,b)$, and hence
$\mathrm{pa17}^{-}$ follows.

Conversely, assume $\mathrm{pa17}^{-}$, and let $x\ne0$ and $y$ be
arbitrary. We show $\operatorname{euc}(x,y)$. If $y=0$, take both the
quotient and the remainder to be $0$. If $x=1$, take the quotient to be $y$
and the remainder to be $0$. If $1<x$ and $0<y$, then
Corollary~\ref{cor:square-transfer} gives
$\operatorname{cop}(x^2,xy+1)$. Hence $\mathrm{pa17}^{-}$ gives
$\operatorname{euc}(x^2,xy+1)$, and the equivalence in the same corollary
gives $\operatorname{euc}(x,y)$. Thus $\mathrm{pa17}$ holds.
\end{proof}

\begin{cor}\label{cor:extensions}
Let $T$ be any extension of $\mathrm{PA}^{-}$. Then
\[
T\vdash\bigl(\mathrm{pa17}^{-}\leftrightarrow\mathrm{pa17}\bigr).
\]
\end{cor}

\begin{proof}
This follows from Corollary~\ref{cor:equivalence} and the monotonicity of
provability under the addition of assumptions.
\end{proof}

\appendix
\setcounter{thm}{0}
\renewcommand{\thethm}{\thesection.\arabic{thm}}
\renewcommand{\thelem}{\thethm}
\renewcommand{\thecor}{\thethm}
\section{Deriving pa14 and pa15 from pa17}

Throughout Appendices A and B, we use standard order consequences of
pa1--pa13; see Je\v{r}\'abek~\cite[Lemma~2]{jerabek} for nonnegativity,
discreteness, additive cancellation, and monotonicity. In particular,
adding a positive element strictly increases a number.

\begin{lem}[A quotient lemma]\label{lem:quotient-one}
In $\mathrm{PA}^{-}_{\mathrm{jer}}$, if
\[
0<A\le B<2A,\qquad B=Aq+r,\qquad r<A,
\]
then $q=1$.
\end{lem}

\begin{proof}
If $q=0$, then $B=r<A$, contradicting $A\le B$. If $q\ne0$, then
$1\le q$; by pa11, either $q=1$ or $2\le q$. In the latter case,
monotonicity and $r\ge0$ give $2A\le Aq\le Aq+r=B$, contradicting
$B<2A$.
\end{proof}

\begin{thm}\label{thm:pa14-from-pa17}
\[
\mathrm{PA}^{-}_{\mathrm{jer}}+\mathrm{pa17}
\vdash \mathrm{pa14}.
\]
\end{thm}

\begin{proof}
The case $x=0$ is immediate, so assume $x>0$. Applying $\mathrm{pa17}$ with
divisor $2$, we obtain
\[
x=2q+r,\qquad r<2.
\]
The axioms pa11 and addition cancellation imply $r\le1$. Applying pa11
once more, either $r=1$, or $r+1\le1$. In the latter case, if $r\ne0$,
then $1\le r$, whence $2\le r+1\le1$, contradicting pa10. Thus
$r=0$ or $r=1$.

If $r=1$, then $x=(2q)+1$, as required. Suppose instead that $r=0$, so
$x=2q$. Since $x>0$, we have $q>0$, and hence $1\le q$. Therefore
\[
q+1\le2q<2(q+1).
\]
The divisor $q+1$ is nonzero, so another application of $\mathrm{pa17}$
gives
\[
2q=(q+1)k+s,\qquad s<q+1.
\]
Lemma~\ref{lem:quotient-one}, applied with divisor $q+1$ and dividend $2q$,
shows that $k=1$. Consequently,
\[
x=2q=(q+1)+s=(q+s)+1.
\]
This proves $\mathrm{pa14}$.
\end{proof}

\begin{cor}\label{cor:pa15-from-pa17}
\[
\mathrm{PA}^{-}_{\mathrm{jer}}+\mathrm{pa17}
\vdash \mathrm{pa15}.
\]
\end{cor}

\begin{proof}
By Theorem~\ref{thm:pa14-from-pa17}, $\mathrm{pa14}$ follows from
$\mathrm{PA}^{-}_{\mathrm{jer}}+\mathrm{pa17}$. Visser~\cite[Theorem~3.7]{visser} shows that
$\mathrm{PA}^{-}_{\mathrm{jer}}+\mathrm{pa14}+\mathrm{pa17}$ proves
$\mathrm{pa15}$. Combining these results gives the claim.
\end{proof}

\section{Restricted and Full Division over the Base Theory}

\begin{thm}\label{thm:pa15-from-pa17-minus}
\[
\mathrm{PA}^{-}_{\mathrm{jer}}+\mathrm{pa17}^{-}
\vdash\mathrm{pa15}.
\]
\end{thm}

\begin{proof}
Work in $\mathrm{PA}^{-}_{\mathrm{jer}}+\mathrm{pa17}^{-}$ and assume
$x\le y$. We use the order consequences of pa1--pa13 stated above.
In particular, $u<v$ iff $u+1\le v$,
and $u+c\le v+c$ implies $u\le v$.

Put
\[
a_0=x+y+2,\quad b_0=2y+2,\qquad a_1=a_0+1,\quad b_1=b_0+1,
\]
and let
\[
A_{i}=a_{i}^2,\qquad B_{i}=a_{i}b_{i}+1,\qquad D_{i}=b_{i}^2.
\]
Define the following terms:
\[
C_0=2xy+2x+2y^2+6y+3,\qquad
C_1=2xy+3x+2y^2+9y+8.
\]
The definitions give $C_{i}+1=a_{i}b_{i}$, so $B_{i}=C_{i}+2$ and
\[
A_{i}D_{i}={(a_{i}b_{i})}^2={(C_{i}+1)}^2=(C_{i}+2)C_{i}+1=B_{i}C_{i}+1.
\]
Thus $\operatorname{cop}(A_{i},B_{i})$ holds. Applying $\mathrm{pa17}^{-}$,
choose $q_{i},r_{i}$ such that
\[
B_{i}=A_{i}q_{i}+r_{i},\qquad r_{i}<A_{i}.
\]
Since $x\le y$, we have $a_{i}\le b_{i}$, so
\[
A_{i}=a_{i}^2\le a_{i}b_{i}\le B_{i}.
\]
Set $k_0=2x+2$ and $k_1=2x+3$. Then $2a_{i}=b_{i}+k_{i}$, and hence
\[
2A_{i}+1=B_{i}+a_{i}k_{i}.
\]
Since $a_{i}\ge2$ and $k_{i}\ge2$, we have $1<4\le a_{i}k_{i}$. Therefore
\[
B_{i}+1<B_{i}+a_{i}k_{i}=2A_{i}+1,
\]
so $B_{i}<2A_{i}$. Thus
$0<A_{i}\le B_{i}<2A_{i}$.
Lemma~\ref{lem:quotient-one} gives $q_{i}=1$, hence
\begin{equation}
B_{i}=A_{i}+r_{i}\qquad(i=0,1).
\label{eq:appendix-b-remainders}
\end{equation}

The definitions give the identity
\[
B_0(a_0+1)+A_1a_0
=B_1a_0+A_0(a_0+1)+1.
\]
Substituting~\eqref{eq:appendix-b-remainders} yields
\[
r_0(a_0+1)=r_1a_0+1.
\]
Thus $\operatorname{cop}(r_0,r_1)$ holds. A further application of
$\mathrm{pa17}^{-}$ gives $q,s$ such that
\[
r_1=r_0q+s,\qquad s<r_0.
\]

Here,
\[
\begin{aligned}
B_0+A_1&=a_0b_0+a_0^2+2a_0+2,\\
B_1+A_0&=a_0b_0+a_0^2+a_0+b_0+2.
\end{aligned}
\]
Since $a_0\le b_0$, we obtain $B_0+A_1\le B_1+A_0$. Substituting
~\eqref{eq:appendix-b-remainders} gives
$r_0\le r_1$. We also have the identity
\[
2B_0+A_1+x^2+x
=B_1+2A_0+y^2+y+1.
\]
Since $x\le y$, monotonicity gives $x^2+x\le y^2+y$. Adding
$B_1+2A_0+1$ to both sides and using the identity above gives
\[
B_1+2A_0+1+x^2+x\le2B_0+A_1+x^2+x,
\]
so
\[
B_1+2A_0+1\le2B_0+A_1.
\]
Substituting~\eqref{eq:appendix-b-remainders} gives
\[
r_1+1\le2r_0,
\]
so $r_1<2r_0$. Applying Lemma~\ref{lem:quotient-one} to the division
$r_1=r_0q+s$, with $0<r_0\le r_1<2r_0$, gives $q=1$.
Consequently,
\[
r_1=r_0+s.
\]

Finally, the definitions give
\[
B_1=B_0+a_0+b_0+1,\qquad A_1=A_0+2a_0+1.
\]
Combining these equations with~\eqref{eq:appendix-b-remainders} and
$r_1=r_0+s$ yields
\[
b_0=a_0+s.
\]
By the definitions of $a_0,b_0$, this is
\[
2y+2=x+y+2+s,
\]
which gives $y=x+s$. This proves pa15.
\end{proof}

\begin{cor}\label{cor:equivalence-over-jer}
\[
\mathrm{PA}^{-}_{\mathrm{jer}}\vdash
\bigl(\mathrm{pa17}^{-}\leftrightarrow\mathrm{pa17}\bigr).
\]
\end{cor}

\begin{proof}
Theorem~\ref{thm:pa15-from-pa17-minus} shows that
$\mathrm{PA}^{-}_{\mathrm{jer}}+\mathrm{pa17}^{-}$ proves pa15. Hence, under
$\mathrm{pa17}^{-}$, the axioms of $\mathrm{PA}^{-}$ are available, and
Corollary~\ref{cor:equivalence} gives $\mathrm{pa17}$.

Conversely, suppose $\mathrm{pa17}$ and let $ad=bc+1$ witness
$\operatorname{cop}(a,b)$. If $a=0$, then $ad=0$, so $bc+1=0$. But
pa1--pa13 imply $0\le bc$, and pa12 then gives $1\le bc+1$, contradicting
pa10. Thus $a\ne0$,
and $\mathrm{pa17}$ yields $\operatorname{euc}(a,b)$. Therefore
$\mathrm{pa17}^{-}$ holds.
\end{proof}

\section*{Acknowledgements}

The author thanks Albert Visser for helpful correspondence and for
encouraging the author to write up this result.

\section*{Use of Generative AI}

Prism's AI assistant helped check calculations and translate the manuscript
into English. The author reviewed its output and takes responsibility for
the final text.

\end{document}